\documentclass[11pt]{article}
\usepackage{amsmath}
\usepackage{amssymb}
\usepackage{amsthm}
\usepackage{amsfonts}
\usepackage{url}
\usepackage{hyperref}
\usepackage{appendix}
\usepackage{verbatim}
\usepackage{graphicx,color}

\input xy
\xyoption{all}

\theoremstyle{plain}
\newtheorem{theorem}{Theorem}[section]
\newtheorem{proposition}[theorem]{Proposition}
\newtheorem{lemma}[theorem]{Lemma}
\newtheorem{corollary}[theorem]{Corollary}
\newtheorem{conjecture}[theorem]{Conjecture}

\newtheorem*{claim*}{Claim}

\newcommand\BeginClaimProof{\begin{proof}[{Proof of Claim \theclaim}]}
\newcommand\EndClaimProof{\end{proof}}
\newtheorem{letterthm}{Theorem}
\newtheorem{lettercor}[letterthm]{Corollary}

\theoremstyle{definition}
\newtheorem{definition}[theorem]{Definition}

\newtheorem{question}[theorem]{Question}

\theoremstyle{remark}
\newtheorem*{remark}{Remark}

\DeclareMathOperator{\Aut}{Aut}
\DeclareMathOperator{\Out}{Out}
\DeclareMathOperator{\Mod}{Mod}
\DeclareMathOperator{\GL}{GL}

\newcommand\D{\Delta}

\newcommand\wh{\widehat}
\newcommand\Z{\mathbb Z}
\newcommand\G{\Gamma}
\newcommand\Q{\mathbb Q}

\newcommand{\punc}{\Sigma_{1,1}}

\def\-{\overline}

\def\wh{\widehat}
\def\g{\gamma}
\def\G{\Gamma}

\def\Z{\Bbb Z}
\def\Q{\Bbb Q}

\def\F{\Bbb F}

\def\e{\varepsilon}

\def\ker{\rm{ker}}

\def\PSL{\rm{PSL}(2,\Z)}

\def\GL{\rm{GL}(2,\Z)}
\def\PGL{\rm{PGL}(2,\Z)}

\def\Tor{\rm{Tor}}

\def\qed{ $\sqcup\!\!\!\!\sqcap$}

\def\tr{\mbox{\rm{tr}}\, }

\def\G{\Gamma}

\def\<{\langle}
\def\>{\rangle}

\def\Gp{\Gamma_{\phi}}

\newcommand{\Addresses}{{
  \bigskip
  \footnotesize

  M.~Bridson, \textsc{Mathematical Institute, Andrew Wiles Building, University of Oxford, Oxford OX2 6GG, UK}\par\nopagebreak
  \textit{E-mail address}, M.~Bridson: \texttt{bridson@maths.ox.ac.uk}

  \medskip

  A.~Reid, \textsc{Department of Mathematics, University of Texas, 1 University Station C1200, Austin, TX 78712-0257, USA}\par\nopagebreak
  \textit{E-mail address}, A.~Reid: \texttt{areid@math.utexas.edu}

  \medskip

  H. Wilton, \textsc{DPMMS, Centre for Mathematical Sciences, Wilberforce Road, Cambridge CB3 0WB, UK}\par\nopagebreak
  \textit{E-mail address}, H.~Wilton: \texttt{h.wilton@maths.cam.ac.uk}

}}

\begin{document}

\date{13 July 2017}

\title{Profinite rigidity and surface bundles over the circle}
\author{Martin R. Bridson, Alan W. Reid and Henry Wilton}

\maketitle

\begin{abstract} If $M$ is a compact 3-manifold whose first betti number is 1, and $N$ is a compact 3-manifold such that $\pi_1N$ and $\pi_1M$ have the same finite quotients, then $M$ fibres over the circle if and only if $N$ does. We prove that groups of the form $F_2\rtimes\Z$ are distinguished from one another by their profinite completions.
Thus, regardless of betti number, if $M$ and $N$ are punctured torus bundles  over the circle and $M$ is not homeomorphic to $N$, then there is a finite group $G$ such that one of $\pi_1M$ and $\pi_1N$ maps onto $G$ and the other does not. 
\end{abstract}

\section{Introduction}

When one wants to understand a finitely presented group it is natural to explore
its finite quotients, and this is a well-trodden path in many contexts.
For instance, one might try to prove that a
presentation does not represent the trivial group by exhibiting
a map onto a non-trivial finite group, or one might try to prove that
two groups are not isomorphic by counting maps to small finite groups. 
The potential of such techniques depends on the extent to which the groups being
studied are determined by the totality of their finite quotients. If 
the groups $\G$ at hand are residually finite, i.e.\ every finite subset injects into some finite quotient, 
then it is reasonable to expect that one will be able to detect many properties of $\G$ from
the totality of its finite quotients. 

Attempts to lend precision to this observation,
and to test its limitations, have surfaced repeatedly 
over the last forty years. There has been a particular resurgence of interest in recent years in the
context of low-dimensional topology, where the central problem is that of 
distinguishing between compact 3-manifolds $M$ and $N$ by
finding a finite quotient of $\pi_1M$ that is not a quotient of
$\pi_1N$. In more sophisticated terminology, one wants to develop a complete understanding of the
circumstances in which 
fundamental groups of non-homeomorphic manifolds $M$ and $N$ can have
isomorphic \emph{profinite completions} $\wh{\pi_1M}$ and
$\wh{\pi_1N}$. (The profinite completion $\wh{G}$ of a discrete
group $G$ is the inverse limit of the inverse system of finite quotients of $G$.)

There has been a good deal of progress on this question recently:
Boileau and Friedl \cite{boileau_profinite_2015}
proved, among other things, that for closed 3-manifolds with $H_1(M,\Z)\cong\Z$,
being fibred is an invariant of the profinite completion; using very different methods,
Bridson and Reid \cite{bridson_profinite_2015} proved that if $M$ is a compact manifold with
non-empty boundary that fibres and has first betti number 1, 
and if $N$ is a compact 3-manifold with $\wh{\pi_1N}\cong\wh{\pi_1M}$, then $N$ has non-empty boundary and
fibres; and if $\pi_1M$ has the form $F_r\rtimes\Z$, with $F_r$ free of rank $r$, then so does $\pi_1N$
(but we do not know if the actions of $\Z$ on $F_r$ can be different).
It follows, for example, that the complement of the 
figure-8 knot is distinguished from all other 3-manifolds by the
profinite completion of its fundamental group \cite{boileau_profinite_2015,bridson_profinite_2015}. 

In the negative direction, Funar \cite{funar_torus_2013}
pointed out that old results of Stebe \cite{stebe_conjugacy_1972b} imply that torus bundles over the circle with {\rm{Sol}}
geometry cannot, in general, be distinguished from one another by the profinite completions of their fundamental
groups $\Z^2\rtimes\Z$. By adapting arguments of Baumslag \cite{baumslag_residually_1974}, Hempel \cite{hempel_some_2014} exhibited a similar phenomenon among bundles with higher genus fibres and finite monodromy; see also \cite{wilkes_profinite_2015}.
 
In this paper we advance the understanding of profinite rigidity for surface bundles over the 
circle in two ways. First, taking up the theme of \cite{bridson_profinite_2015}, we show that in the case
of punctured-torus bundles over the circle, the monodromy of the bundle is determined by the profinite completion 
of the fundamental group and, moreover,
profinite rigidity persists if one drops the hypothesis
$b_1(M)=1$. Secondly, 
we extend the fibring theorems of Boileau--Friedl and
Bridson--Reid to the case of all bundles $M$ with compact fibre and $b_1(M)=1$ (Theorem \ref{closedprofinite}).

To state the first of these results more precisely,
we define $\punc$ to be the once-punctured torus, and for any $\phi$ in the extended mapping class group $\Mod^{\pm}(\punc)\cong{\rm{GL(2,\Z)}}\cong\Out(F_2)$, let $M_\phi$ be the mapping torus of (a homeomorphism representing) $\phi$. Let $F_2$ denote the non-abelian free  group of rank $2$.


\begin{letterthm}\label{thm1}
Let $\phi_1,\phi_2\in\Out(F_2)$ and let $\G_i=F_2\rtimes_{\phi_i}\Z$.
If  $\widehat{\Gamma}_{1}\cong\widehat{\Gamma}_{2}$, then  
$\phi_1$ is conjugate to $\phi_2$ in ${\rm{Out}}(F_2)={\rm{GL}}(2,\Z)$,   hence
$\G_{1}\cong\G_{2}$ and  $M_{\phi_1}$ is homeomorphic to $M_{\phi_2}$.
\end{letterthm}

\begin{lettercor}
Let $M$ be a hyperbolic 3-manifold that fibres over the circle with fibre a one-holed torus, and let $N$ be
a compact connected 3-manifold. If $\wh{\pi_1N}\cong\wh{\pi_1M}$, then $N$ is homeomorphic to $M$. 
\end{lettercor} 

To deduce this corollary, first observe
that  since $M$ is hyperbolic, its first betti number is $1$ and
 $\pi_1M$ has the form $F_2\rtimes_\phi\Z$ with $\phi\in{\rm{GL}}(2,\Z)$ a hyperbolic matrix.
\cite[Theorem B]{bridson_profinite_2015} 
states that $N$ has non-empty boundary,
is a bundle with compact fibre of euler characteristic $-1$, and
$\pi_1N\cong F_2\rtimes_\psi\Z$. 
Theorem \ref{thm1} tells us that  $\psi$, which describes the monodromy of $N$, is conjugate to $\phi$ and
therefore is hyperbolic. The one-holed torus is the only compact surface of Euler characteristic $-1$
that supports a hyperbolic automorphism, so $N$ is a once-holed torus bundle with the same mondromy as $M$,
and hence $N\cong M$. 

Although our results are concrete, the key facts that we exploit
are abstract properties of mapping class groups.  For Theorem \ref{thm1}
we use the fact that $\Mod^\pm(\punc)$ is \emph{omnipotent}
and enjoys the \emph{congruence subgroup property} (see Section 2).  Corresponding
results for mapping class groups of surfaces of higher
complexity are beyond the reach of current techniques. However, if one
assumes those properties of mapping class groups, then one can obtain
similar results for bundles with higher-genus fibre (see Theorem
\ref{generalthm}); our proof of Theorem \ref{thm1} is presented in a manner that emphasizes this
general strategy.

Our other main result completes one step in the strategy by establishing that fibring is
a profinite invariant for manifolds with first betti number 1: this is achieved by combining
Theorem \ref{closedprofinite} with the corresponding result in the case of manifolds with
boundary \cite{bridson_profinite_2015}.

\begin{letterthm}
\label{closedprofinite}
Let $M$ be a closed orientable hyperbolic 3-manifold with first betti number $b_1(M)=1$ that is a 
bundle with fibre
a closed surface $\Sigma$ of genus $g$. Let $N$ be a compact 3-manifold with
$\widehat{\pi_1(N)} \cong \widehat{\pi_1 (M)}$. Then $N$ is also
a closed orientable hyperbolic 3-manifold with $b_1(N)=1$ that is a 
bundle with fibre a closed surface
of genus $g$.
\end{letterthm}

In \cite{boileau_profinite_2015},  this theorem was proved under the assumption that $H_1(M,\Z)\cong\Z$, using different methods: we avoid their use of twisted Alexander polynomials, relying instead on topological arguments.

Throughout, we assume that the reader is familiar with elementary facts about profinite groups, as described in
\cite{ribes_profinite_2010} for example.

\subsection*{Acknowledgements}

We thank Pavel Zalesskii for drawing to our attention an argument that appears in the proof of Theorem \ref{closedprofinite}.
We also thank the Mathematical Sciences Research Institute for their hospitality,  and the anonymous referee for constructive comments which improved our exposition.
The first author was
supported in part by grants from the                        
EPSRC and a Royal Society Wolfson Merit Award. The second
author was supported in part by an NSF grant and The Wolfensohn Fund.
He would also like to thank the Institute for Advanced Study for its
hospitality whilst this work was completed. The third author was supported
by a grant from the EPSRC.

\section{Congruence omnipotence} \label{ss: Congruence omnipotence}

In this section we define the notion of {\em congruence omnipotence} in
$\Out(G)$, for $G$ a finitely generated group. Our main theorem
(Theorem \ref{generalthm}) asserts that, when it holds, congruence
omnipotence enables one to deduce profinite rigidity results for
mapping tori $G\rtimes\Z$.


\begin{definition} Let $G$ be a finitely generated group and let $H\subseteq \Out(G)$ be a subgroup.
A finite quotient  $H\to Q$ is a \emph{$G$-congruence quotient} if it factors through $\pi: H\to P
\subset\Out(G/K)$ where $K$ is a characteristic subgroup of finite index in $G$ and
$\pi$ is the restriction of the natural map $\Out(G)\to \Out(G/K)$.
 We say that \emph{$\Out(G)$ has the congruence subgroup property} if every finite quotient of $\Out(G)$ is a $G$-congruence quotient. More generally, we say that a subgroup $H\subseteq \Out(G)$ has the
{\em $G$-congruence subgroup property} if every finite quotient of $H$  is a $G$-congruence quotient.
\end{definition}

\begin{remark}  Care is needed in the above definition:  
  there may be distinct groups $G_1$ and $G_2$ with $\Out(G_1)\cong\Out(G_2)$ such that every finite quotient is congruence with respect to $G_1$ but not with respect to $G_2$. For instance, this phenomenon occurs with $\Out(F_2)\cong GL(2,\Z)=\Out(\Z^2)$, which has the congruence subgroup property with respect to $F_2$ but not $\Z^2$. 
  Thus ``$\Out(G)$ has the congruence subgroup property" is a statement about $G$ and not 
  the abstract group $\Out(G)$. 
\end{remark}


\emph{Omnipotence} was first defined by Wise in the context of free and hyperbolic groups.   

\begin{definition}\label{d:omnip}
Let $\G$ be a group. Elements $\g_1,\g_2\in \G$ of infinite order are said to be {\em{independent}} if no non-zero power of $\g_1$ is conjugate to a non-zero power of $\g_2$ in $\G$. An $m$-tuple $(\gamma_1,\ldots,\gamma_m)$ of elements  is  \emph{independent} if $\g_i$ and $\g_j$ are independent whenever $1\le i < j\le m$.  The  group $\Gamma$ is said to be \emph{omnipotent} if, for every independent $m$-tuple $(\gamma_1,\ldots,\gamma_m)$ of elements in $\G$,   there exists a positive integer $\kappa$ such that, for every $m$-tuple of positive integers $(e_1,\ldots,e_m)$ there is a homomorphism to a finite group
\[
q:\Gamma\to Q
\]
such that $o(q(\gamma_i))=\kappa e_i$ for $i=1,\dots,m$, where $o(g)$ denotes the order of a group element $g$.  For a subgroup $H$ of $\Gamma$, if we wish to emphasize that an $m$-tuple of elements is independent in $H$, we will say that the tuple is \emph{$H$-independent}.
\end{definition}

We focus on a more restrictive form of omnipotence that is adapted to our purposes.  
Two motivating examples that we have in mind are: (i) where $G$ is a closed surface group or a free group
and $H=\Out(G)$;
and (ii) where $G$ is a free
group and $H\subseteq\Out(G)$ is the mapping class group of a punctured surface.  In these 
contexts, there is usually a favoured class of elements for which one expects omnipotence to hold,
e.g. pseudo-Anosovs in the case of mapping class groups, or fully irreducible elements in the case of $\Out(F_n)$; we therefore work with subsets $\mathcal{S}\subseteq H$. We also insist that the finite quotients obtained should be $G$-congruence quotients.

\begin{definition}
Let $G$ be a finitely generated group, let $H$ be a subgroup of $\Out(G)$ and let $\mathcal{S}$ be a subset of $H$.  We say\footnote{if $H=\Out(G)$ we abbreviate this to ``$G$-congruence omnipotent''}
 that $\mathcal{S}$ is $(G,H)$-\emph{congruence omnipotent} if, for every $m$ and every $H$-independent $m$-tuple $(\phi_1,\ldots,\phi_m)$ of elements of $\mathcal{S}$, there is a constant $\kappa$ such that, for any $m$-tuple of positive integers $(n_1,\ldots,n_m)$, there is a $G$-\emph{congruence} quotient $q:H\to Q$ such that $o(q(\phi_i))=\kappa n_i$ for all $i$.
\end{definition}

\begin{remark}
If $\Out(G)$ is omnipotent and has the congruence subgroup property, then the set of infinite-order elements is $G$-congruence omnipotent in $\Out(G)$.
\end{remark}

Our most general theorem shows how congruence omnipotence can be used as a tool
for establishing profinite rigidity for the mapping tori associated to automorphisms of a fixed group $G$.

\begin{theorem}\label{generalthm}
Let $G$ be a finitely generated group, let $H\subseteq\Out(G)$ be a subgroup
 and let $\mathcal{S}$ be a $(G,H)$-congruence omnipotent subset.  Let $\phi_1,\phi_2\in\mathcal{S}$, let $\Gamma_i=G\rtimes_{\phi_i}\Z$ and suppose that $b_1(\Gamma_{i})=1$ for $i=1,2$.  If $\wh{\Gamma}_{1}=\wh{\Gamma}_2$ then there is an integer $n$ such that $\phi_1^n$ is conjugate in $H$ to $\phi_2^{\pm n}$.
\end{theorem}
 
The key observation in the proof of Theorem \ref{generalthm}  is contained in the following lemma.

\begin{lemma}\label{lemma2}
Let $\Gamma_i=G\rtimes_{\phi_i}\Z$ for $i=1,2$.  If $\widehat{\Gamma}_1\cong\widehat{\Gamma}_2$ and $b_1(\Gamma_i)=1$ for $i=1,2$, then the image of
$\phi_1$ and $\phi_2$  generate the same cyclic subgroup
in the outer automorphism group of any characteristic quotient of $G$.
\end{lemma}

\begin{proof} 
We fix an identification $\widehat{\Gamma}_1=\widehat{\Gamma}_2$. 
The unique epimorphism $\Gamma_i\to\Z$ defines a short exact sequence
$$
1\to \wh{G}\to \wh{\G}_i\to \wh{\Z}\to 1.
$$
If $K<G$ is a characteristic subgroup of finite index, then the canonical map $G\to G/K$ defines 
an epimorphism $\wh{G}\to G/K$. Since $\wh{K}$ is normal in $\wh{\Gamma}_i$, the action of $\wh{\Z}$ 
on $\wh{G}$ induced by conjugation in $\wh{\G}_i$ descends to an
action on $\wh{G}/\wh{K}=G/K$, defining a cyclic subgroup  $C<\Out(G/K)$, of order $m$ say. The righthand factor
of $\Gamma_i=G\rtimes_{\phi_i}\Z$ is dense in $\wh{\Z}$, so the image of $\phi_i$ generates $C$ for $i=1,2$.
\end{proof}

\begin{proof}[Proof of Theorem \ref{generalthm}]
Suppose first 
that $\phi_1$ and $\phi_2$ are $H$-independent.  Since $\mathcal{S}$ is $(G,H)$-congruence omnipotent, there exists a $G$-congruence quotient $q:H\to Q$ such that $o(q(\phi_1))\neq o(q(\phi_2))$.  By definition, $q$ factors as
\[
H\to P \to Q
\]
with $H\to P$ the restriction of the natural map  $\Out(G)\to\Out(G/K)$ for some characteristic subgroup $K$ of finite index in $G$.
Since the images of $\phi_1$ and $\phi_2$ have distinct orders in $Q$, 
the images of $\phi_1$ and $\phi_2$
cannot generate the same cyclic subgroup of $\Out(G/K)$, contradicting 
Lemma \ref{lemma2}.

Therefore, $\phi_1$ and $\phi_2$ are not $H$-independent, so there are positive integers $n_1$ and $n_2$ such that $\phi_1^{n_1}$ is conjugate to $\phi_2^{\pm n_2}$ in $H$. It remains to prove that $n_1=n_2$.  By congruence omnipotence applied to the $1$-tuple $(\phi_1)$, there is a characteristic subgroup $K$ of finite index in $G$ such that $n_1n_2$ divides $o(q(\phi_1))$, where $q:\Out(G)\to\Out(G/K)$ is the natural homomorphism. In particular, we have
\[
o(q(\phi_1))/n_1=o(q(\phi_1^{n_1}))=o(q(\phi_2^{n_2}))=o(q(\phi_2))/n_2
\] 
and so, since Lemma \ref{lemma2} implies that $o(q(\phi_1))=o(q(\phi_2))$, we have $n_1=n_2$ as claimed.
\end{proof}

\subsection{$\Out(F_2)$ is congruence omnipotent} 

The congruence subgroup property for $\Out(F_2)$ was established by Asada \cite{asada_faithfulness_2001};
alternative proofs were given by Bux--Ershov--Rapinchuk \cite{bux_congruence_2011} and  Ellenberg--McReynolds \cite{ellenberg_arithmetic_2012}.

\begin{theorem}[Asada \cite{asada_faithfulness_2001}]\label{CSP}
For any finite quotient $\Out(F_2)\to Q$ there is a characteristic finite-index subgroup $K$ of $F_2$ such that the quotient map factors as
\[
\Out(F_2)\to\Out(F_2/K)\to Q~.
\]
\end{theorem}
 
Wise proved that finitely generated free groups are omnipotent and later extended his proof to virtually special groups \cite{wise_subgroup_2000}.
Bridson and Wilton gave a more direct proof that virtually free groups are omnipotent \cite{bridson_triviality_2015}.
As $\Out(F_2)$ is virtually free, in the light of Theorem \ref{CSP} we have:

\begin{proposition}\label{CSPprop}
The set of elements of infinite order in $\Out(F_2)$ is $F_2$-congruence omnipotent.
\end{proposition}

\section{Profinite rigidity for punctured-torus bundles}\label{s:F2Z} 

Our proof of Theorem \ref{thm1} relies on a number of elementary calculations in ${\rm{GL}}(2,\Z)\cong
\Out(F_2)$; we have relegated these to an appendix, so as not to disturb the flow of our main argument.
The reader may wish to read that appendix before proceeding with this section.   

\subsection{Reducing to the hyperbolic case}

\begin{theorem}\label{t:pf2conj}
Let $\phi_1,\phi_2\in\Out(F_2)$, let $\Gamma_i=F_2\rtimes_{\phi_i}\Z$ and suppose that $\wh\G_{1}\cong\wh\G_{2}$.
\begin{enumerate}
\item If $\phi_1$ is hyperbolic then $\phi_2$ is hyperbolic.
\item If $\phi_1$ is not hyperbolic, then $\phi_2$ is conjugate\footnote{$\phi$
is always conjugate to $\phi^{-1}$ if $\phi$ is not hyperbolic} to $\phi_1$.
\end{enumerate}
\end{theorem}
\begin{proof} Item (1) was proved in \cite[Proposition 3.2]{bridson_profinite_2015} by arguing $\phi$ is hyperbolic if and only
if $b_1(\G_{\phi^r})=1$ for all $r>0$. (The ``only if" implication follows easily from Lemma \ref{l:H1}.)   We therefore proceed to prove Item (2).

Proposition \ref{p:types} tells us that the list of abelianisations of $\G_\phi$ calculated in Lemma \ref{l:H1}
covers all non-hyperbolic $\phi$. If all of these groups were non-isomorphic then we would be done, but there
remain two ambiguities for which we make special arguments. 

First, to distinguish $\wh\G_{-I}$ 
from $\wh\G_{-U(n)}$ with $n>0$ even, we can invoke Lemma \ref{lemma2}: since $b_1(\G_{-I})=b_1(\G_{-U(n)})=1$,
the automorphisms of $H_1(F_2,\Z/(n+1))$ induced by $-I$ and $-U(n)$ would have the same order if
$\wh\G_{-I}\cong\wh\G_{-U(n)}$, but the former has order $2$ and the latter is
\[
B=\begin{pmatrix} -1& 1\cr 0&  -1\cr\end{pmatrix}
\]
which has order $n+1$.

Second, to distinguish $\wh\G_{\epsilon}$ from $\wh\G_{U(2)}$, we note that $t^2$ is central in
$\G_{\epsilon}=F_2\rtimes_{\epsilon}\<t\>$ and has infinite order in $H_1(\G_{\e},\Z)$, so
if $Q$ is any finite quotient of $\G_{\e}$
then the abelianisation of $Q/Z(Q)$ is a quotient of $\Z\oplus \Z/2 \oplus \Z/2$. On the other
hand,  $F_2\rtimes_{U(2)}\Z$ maps onto the mod-$3$ Heisenberg group
$$
(\Z/3 \oplus \Z/3) \rtimes_{-B}\Z_3,
$$
and the quotient of this group by its centre is $\Z/3 \oplus \Z/3$.
\end{proof}

During the course of the proof of Theorem \ref{thm1}, we will need to argue that $\wh{\Gamma}_\phi\ncong\wh{\Gamma}_{-\phi}$.  The required calculation can be found in \cite[Lemma 3.5]{bridson_profinite_2015}, which we reproduce below for the reader's convenience. 

\begin{lemma}\label{l:torsion}${}$
\begin{enumerate}
\item  $b_1(\Gp)=1$ if and only if $1 + \det \phi \neq  \tr \phi$.
\item If $b_1(\Gp)=1$ then $H_1(\Gp,\Z)\cong\Z\oplus T$, where $|T|= |1+\det\phi - \tr\phi|$.
\end{enumerate}
\end{lemma}

\begin{proof} By choosing a representative $\phi_*\in {\rm{Aut}}(F_2)$, we get a presentation for $\Gp$,
$$
\< a,b, t \mid tat^{-1} = \phi_*(a),\, tbt^{-1}=\phi_*(b)\>.
$$
By abelianising, we see that $H_1(\Gp,\Z)$ is the direct sum of $\Z$ (generated by the image of $t$)
and $\Z^2$ modulo the image of $\phi-I$. The image of $\phi-I$ has  finite index if and only if $\det(\phi-I)$
is non-zero, and a trivial calculation shows that this determinant is $1 - \tr\phi + \det\phi$. If the index is
finite, then the quotient has order $|\det(\phi - I)|$.
\end{proof}

\subsection{End of the proof: the hyperbolic case}

\begin{proof}[Proof of Theorem \ref{thm1}] By Theorem \ref{t:pf2conj}, we may assume that both $\phi_1$ and $\phi_2$
are hyperbolic.  In particular, $b_1(\G_1)=b_1(\G_2)=1$ and so, by Proposition \ref{CSPprop}, we may invoke Theorem \ref{generalthm} to deduce that there is an $n$ such that $\phi_1^n=\phi_2^{\pm n}$.  It then follows from Lemma \ref{l:roots} that $\phi_1$ is conjugate to $\pm \phi_2^{\pm 1}$.  To remove the possibility that $\phi_1$ is conjugate to $-\phi_2^{\pm1}$ we use Lemma \ref{l:torsion} to compare the order of the torsion
subgroup  in $H_1(\G_{\phi},\Z)$ with that in $H_1(\G_{-\phi},\Z)$, noting that
$\det(-\phi) = \det \phi$ but $\tr(-\phi) = -\tr\phi$.
\end{proof}

\section{Closed hyperbolic bundles with $b_1(M)=1$}

We now shift our attention to {\em closed} 3-manifolds.
Our purpose in this section is to prove Theorem \ref{closedprofinite}.  We therefore assume that $M$ is a closed, orientable, hyperbolic 3-manifold with $b_1M=1$, fibring over the circle with fibre $\Sigma$, and that $N$ is a closed, orientable 3-manifold with $\wh{\pi_1M}\cong\wh{\pi_1N}$.  We have been informed by Boileau and Friedl that the methods of their paper \cite{boileau_profinite_2015} can be used to give a different proof of Theorem \ref{closedprofinite}.  Extending the result to bundles with $b_1(M)>1$ lies beyond the present scope of  both our techniques and theirs.

\subsection{The main argument}

By arguing as in Theorem 4.1 of \cite{bridson_profinite_2015}, we may assume that $N$ is aspherical, closed and orientable. Since the finite abelian quotients of $\pi_1N$ coincide with those of $\pi_1M$, we also see that $b_1(N)=1$.  And by \cite{wilton_distinguishing_2017} we know that $N$ is hyperbolic.

Set $\Delta=\pi_1(N)$ and $\Gamma=\pi_1(M)$.  There is a unique epimorphism $\Delta\rightarrow {\Bbb Z}$, and dual to this we can find a closed embedded non-separating incompressible surface $S\subset N$.  If $S$ is a fibre we will be done, for exactly as in Lemma 3.1 of \cite{bridson_profinite_2015}, we get $\widehat{\pi_1(S)} \cong \widehat{\pi_1(\Sigma)}$, from which it easily follows (by noting, for example, 
that $H_1(\Sigma,{\Bbb Z})$ is determined by $\wh{\pi_1\Sigma}$) that $S$ is homeomorphic to $\Sigma$.

Thus, in order to complete our proof of the theorem, it suffices to derive a contradiction from the assumption
that $S$ is not a fibre. The well-known dichotomy of Bonahon and Thurston \cite{bonahon_bouts_1986} 
implies that if $S$ is not a fibre then it is quasi-Fuchsian.
 
 Let $H=\pi_1(S)<\Delta$, let $G=\pi_1\Sigma$, let $K<\Delta$ be the 
 kernel of the unique epimorphism $\Delta\to\Z$,  let $\overline{K}$ denote the closure of
 $K$ in $\widehat\Delta$, and note that $H<K$. 
It is elementary to see that $\G$ induces the full profinite topology on $G$
(see \cite[Lemma 2.2]{bridson_profinite_2015} for example),
and it follows from Agol's virtually special theorem \cite{agol_virtual_2013} that, since $H$ is quasiconvex and hence a virtual retract, the full profinite topology is induced on $H$ (see, for example,  \cite[(L.16), p.\ 120]{aschenbrenner_manifold_2015}).
 
The isomorphism $\widehat\G\cong\widehat\Delta$ identifies $\widehat G$ with $\overline K$
 (each being the kernel of the unique epimorphism $\widehat\G\to\widehat\Z$). Thus
 $\widehat H = \overline H$ is a subgroup of $\widehat G$.
 To complete
 the proof of the theorem, we need two lemmas. The first of these lemmas
 relies on Agol's theorem  \cite{agol_virtual_2013}, while the second  
  is based on a standard exercise about duality groups at a prime $p$ that was drawn to our attention by P. Zalesskii.

\begin{lemma}\label{freeonfiniteindex}
There is a finite index subgroup $\Delta_0<\Delta$
such that:
\begin{enumerate}
\item~$H < \Delta_0$ ;
\item there exists an epimorphism $f:\Delta_0\rightarrow F_2$ (a free
group of rank $2$) such that $H<\ker~f$).
\end{enumerate}
\end{lemma}

\begin{lemma}
\label{zalesskii}  $[\widehat{G}:\widehat{H}]<\infty$.
\end{lemma}

We defer the proofs of these lemmas for a moment while we complete the proof of the theorem.

Define $\G_0=\G\cap\widehat\D_0$ and  $G_0=G\cap\G_0$. The surjection $\Delta_0\to F_2$
of Lemma \ref{freeonfiniteindex} induces an epimorphism  $\widehat\G_0=\widehat{\Delta}_0 \rightarrow \widehat{F}_2$,
the kernel of which contains $\widehat H$. It follows from Lemma \ref{zalesskii} that the image of $\wh{G}_0$ in 
$\widehat{F}_2$
is finite. But $\widehat{F}_2$ is torsion-free, so in fact the image of $\widehat{G}_0$ must be trivial, which
means that $\widehat{\G}_0=\widehat{\Delta}_0 \rightarrow \widehat{F}_2$ factors through the abelian
group $\widehat{\G}_0/\widehat{G}_0\cong\widehat{\Z}$, which is impossible. This contradiction
completes the proof of Theorem \ref{closedprofinite}.
\qed

\subsection{Proofs of lemmas}



\begin{proof}[Proof of Lemma \ref{freeonfiniteindex}]
We first argue that there are infinitely many double cosets $H\backslash\Delta/H$.  The group $\Delta$ acts on the Bass--Serre tree $T$ corresponding to the splitting of $\Delta$ obtained by cutting $N$ along $S$.  Let $e$ be the edge of $T$ stabilized by $H$.   The set of double cosets $H\backslash\Delta/H$ is in bijection with the orbits of the edges of $T$ under the action of $H$ on $T$.   Since $H$ acts on $T$ by isometries and there are edges of $T$ at arbitrarily large distance from $e$, it follows that $H\backslash\Delta/H$ is infinite.

By \cite{agol_virtual_2013}, $\Delta$ is virtually special.  Combining the results of \cite{haglund_special_2008} and \cite{minasyan_separable_2006}, the double cosets in $H\backslash\Delta/H$ are separable. Hence there exists a subgroup $\Delta_0$ of finite index in $\Delta$, containing $H$, so that $|\Delta_0\backslash\Delta/H|\geq 4$.   Let $N_0$ be the covering space of $N$ corresponding to $\Delta_0$. Then the complete preimage $S_0\subseteq N_0$ of the surface $S$ 
is embedded, and the components of $S_1,\ldots,S_k$ of $S_0$ naturally correspond to the double cosets $\Delta_0\backslash \Delta/H$. If $S_1$ is the component corresponding to the trivial double coset $\Delta_0H=\Delta_0$, then $S_1$ is homeomorphic to $S$, since $\Delta_0$ contains $H$.  Choose three components $S_2$, $S_3$, $S_4$ of $S_0$, each distinct from $S_1$.  Let $X$ be the dual graph to the decomposition of $N_0$ obtained by cutting along $S_2\cup S_3\cup S_4$.  Then $X$ has three non-separating edges, and hence the fundamental group $F$ of $X$ is free and non-abelian.  Finally, $H$ is in the kernel of the natural epimorphism $q:\Delta_0\to F$, since $q$ is induced by a continuous map $N_0\to X$ that crushes $S_1$ to a vertex.
\end{proof}

We now turn to Lemma \ref{zalesskii}.  We are grateful to Pavel Zalesskii for drawing
 our attention to \cite[ p.\ 44, Exercise 5(b)]{serre_galois_1997}, which guides the proof.
\begin{proof}[Proof of Lemma \ref{zalesskii}]
Suppose for a contradiction that $[\wh{G}_0:\wh{H}]=\infty$.  Since $\widehat{H}$ is closed, we may choose nested finite-index subgroups $U_i$ in $G_0$ so that the intersection $\bigcap_i \widehat{U}_i=\widehat{H}$.

We fix a prime $p$, 
consider the map $f:\Delta_0\to F_2$ provided by Lemma \ref{freeonfiniteindex}, and let $\hat{f}:\wh{\Delta}_0\to\wh{F}^{(p)}_2$ be the composition of the induced map on profinite completions and the
projection from $\wh{F}_2$ to the pro-$p$ completion of $F_2$.  Since $\wh{F}^{(p)}_2$ is non-abelian, $\hat{f}$ certainly does not factor through the quotient map $\wh{\Delta}_0\to\wh{\Delta}_0/\wh{G}_0\cong\wh{\Z}$, so 
the closed subgroup $\wh{L}:=\hat{f}(\wh{G}_0)$ is non-trivial.  Choose an infinite nested sequence of open subgroups $V_i\subseteq\wh{L}$ with trivial intersection and let $W_i:=\wh{U}_i\cap \hat{f}^{-1}(V_i)$. Then
$\bigcap_i {W}_i=\widehat{H}$ and $p$ divides the index $[{W}_i:{W}_{i+1}]$ for infinitely many $i$, so, passing to a subsequence, we may assume that $p$ divides $[{W}_i:{W}_{i+1}]$ for all $i$.

The end of the argument is a standard exercise about duality groups at the prime $p$ (cf.\ \cite[p.\ 44, Exercise 5(b)]{serre_galois_1997}).  We consider continuous cohomology with coefficients in the finite field $\F_p$.  
As a finite-index subgroup of a surface group, each $W_i\cap G_0$ is a 
 surface group, hence it is good in the sense of Serre, which means that each
 of the restriction maps $H^2({W}_i,\F_p)\to H^2({W}_j,\F_p)\cong \F_p$ is multiplication by $[W_i:W_j]$.  Since
\[
H^2(\wh{H},\F_p)=H^2\left(\bigcap_i {W}_i,\F_p\right)=\underrightarrow{\lim}~H^2({W}_i,\F_p)~
\]
 and $p$  divides $[W_i:W_{i+1}]$, we conclude that $H^2(\wh{H},\F_p)=0$, which is a contradiction, since $H$ is
 also a surface group.
\end{proof}

\section{Surfaces of higher complexity and free groups of higher rank}

As far as we know, the hypotheses of Theorem \ref{generalthm} may hold in very great generality. In the general context of mapping class groups, the congruence subgroup property is open, as is omnipotence for pseudo-Anosov elements.  Likewise, in the context of outer automorphism groups of free groups, the congruence subgroup property is open, as is omnipotence for hyperbolic automorphisms. 


\begin{question}\label{generalqmod}
Let $\Sigma$ be a surface of finite type. Might the set of pseudo-Anosovs in the mapping class group $\mathrm{Mod}(\Sigma)$ be $\pi_1\Sigma$-congruence omnipotent?
\end{question}

A positive answer to Question \ref{generalqmod} would have significant ramifications.
For example, it would immediately imply that if $M$ is a closed
hyperbolic 3-manifold with $b_1(M)=1$ and $N$ is a compact 3-manifold with
$\wh{\pi_1M}\cong \wh{\pi_1N}$ then $M$ and $N$ share a common
finite cyclic cover (of the same degree over $N$ and $M$); in particular they are cyclically
commensurable.

The closedness hypothesis assures that the manifolds have homeomorphic fibres. 
Less obviously, if $M$ and $N$ are hyperbolic knot complements in $S^3$ (or in an integral homology sphere)
with $\wh{\pi_1M}\cong\wh{\pi_1N}$, then \cite[Theorem 7.2]{bridson_profinite_2015}  implies
that the fibres are homeomorphic, so again a positive answer to Question 5.1 would imply
that then $M$ and $N$ share a common
finite cyclic cover.  These observations gain further interest in the context
of the following conjecture of the second author and G. Walsh \cite{reid_commensurability_2008}.

\begin{conjecture}
\label{3knots}
Let $K\subset S^3$ be a hyperbolic knot. There are at most $3$ distinct knot
complements in the commensurability class of $S^3\setminus K$.\end{conjecture}

Conjecture \ref{3knots} was proved  in \cite{boileau_knot_2012} in the ``generic case", namely when $K$ has {\em no hidden symmetries} (see  \cite{reid_commensurability_2008} or \cite{boileau_knot_2012} for the definition of hidden symmetry). At present the only knots that are known to have hidden symmetries are the figure-eight knot and the two dodecahedral knots of Aitchison and Rubinstein \cite{aitchison_combinatorial_1992}.  The dodecahedral knots are known to be the only knots in their commensurability class \cite{hoffman_knot_2014}, and their fundamental groups are distinguished by their profinite completions 
using 
\cite{bridson_profinite_2015}, since
one is fibred and the other is not. Since the figure-eight knot group is distinguished from all 3-manifold groups by its profinite completion, the proviso concerning hidden symmetries in the following result may be unnecessary.

\begin{proposition}\label{profinite_knots1}
Let $K\subset S^3$ be a fibred hyperbolic knot. 
If Question \ref{generalqmod} has a positive answer and $K$ has no hidden symmetries, 
then there is no other hyperbolic knot $K'$
such that $\pi_1(S^3\smallsetminus K)$ and $\pi_1(S^3\smallsetminus K')$ have the
same profinite completion. 
\end{proposition}

\begin{proof}~If there were such a $K'$, then by \cite[Theorem
  7.2]{bridson_profinite_2015}  $K'$ would be fibred with fibre of the same genus. 
  A positive answer to Question \ref{generalqmod} would imply that the complements of $K$ and $K'$
  had a common finite cyclic cover of the same degree (in the light of Theorem 2.4).
  In particular the knot groups would
  be commensurable and the complements would have the same volume. But Theorem 1.7 of \cite{boileau_knot_2012} shows that the complements in the commensurability class of a hyperbolic knot that has no hidden symmetries each have
 a different volume.
\end{proof}


\begin{corollary}
\label{profinite_knots2}
Let $K\subset S^3$ be a hyperbolic knot that admits a Lens Space surgery.
If Question \ref{generalqmod} has a positive answer and $K$ has no hidden symmetries, 
then there is no other hyperbolic knot $K'$
such that $\pi_1(S^3\smallsetminus K)$ and $\pi_1(S^3\smallsetminus K')$ have the
same profinite completion. 
\end{corollary}

\begin{proof}
The result follows from Proposition \ref{profinite_knots1} on noting that
Y. Ni \cite{ni_knot_2007} proved that
a (hyperbolic) knot that admits a Lens Space surgery is fibred. 
\end{proof}

\begin{remark}
\noindent (i)~Theorem 1.4 of \cite{boileau_knot_2012} establishes that if the complements of knots without
hidden symmetries are commensurable, then they are actually
cyclically commensurable (in line with our results). 

\smallskip

(ii)~We regard Proposition \ref{profinite_knots1} as giving further credence to the belief that
hyperbolic knot groups are profinitely rigid. This belief is in keeping
with a theme that has recently emerged in low-dimensional topology and Kleinian
groups exploring the extent to which the fundamental group of a finite volume
hyperbolic 3-manifold is determined by the geometry and topology of its
finite covers. An aspect of this is the way that
a ``normalized" version of $|\Tor(H_1(M,{\Bbb Z}))|$ behaves in finite covers;
it is conjectured that this should determine the volume of the manifold. Since 
$\Tor(H_1(M,{\Bbb Z}))$ is detected at the level of the profinite completion,
the volume is thus conjectured to be a profinite invariant.
\end{remark}

Turning to the case of $\Out(F_n)$ we can ask:

\begin{question}\label{generalqout}
Let $F_n$ be the non-abelian free group of rank $n$. Might
the set of fully irreducible automorphisms in $\Out(F_n)$ 
be $F_n$-congruence omnipotent? What about the set of hyperbolic automorphisms?
\end{question}

As above, a positive answer to Question
\ref{generalqout} would imply that hyperbolic mapping tori
$F_n\rtimes\Z$ with $b_1=1$ and the same profinite genus are
commensurable.

\begin{appendices}
\section{Appendix: Computations in $GL(2,\mathbb{Z}$)}

The action of ${\rm{Aut}}(F_2)$ on $H_1(F_2,\Z)$ gives an epimorphism $\Aut(F_2)\to {\rm{GL}}(2,\Z)$ whose kernel  is the group of inner automorphisms. The isomorphism type of $\G_\phi$ depends only on the  conjugacy class of the image of $\phi$ in $\Out(F_2)={\rm{GL}}(2,\Z)$, so we may regard $\phi$ as an element of ${\rm{GL}}(2,\Z)$. We remind the reader that finite-order elements of ${\rm{GL}}(2,\Z)$ are termed {\em{elliptic}},  infinite-order elements with an eigenvalue of absolute value bigger than $1$ are {\em{hyperbolic}}, and the other infinite-order elements are {\em{parabolic}}. 

In this appendix, we collect various standard facts about the algebra of $GL(2,\Z)$.  Each can be checked using elementary algebra (or more elegantly, in some cases, using the action of $\PSL \cong \Z/2\ast\Z/3$
on the dual tree to the Farey tesselation of the hyperbolic plane).  
  The first such fact concerns the uniqueness of roots.

\begin{lemma}\label{l:roots}
If $\phi,\psi\in\GL$ are elements of infinite order and $\phi^n=\psi^n$ for some $n\neq 0$,
then $\phi=\pm \psi$.
\end{lemma}

We next recall the classification of non-hyperbolic elements of $GL(2,\Z)$, up to conjugacy.

\begin{proposition}\label{p:types}
Every non-hyperbolic element of ${\rm{GL}}(2,\Z)$ is conjugate to exactly one of the following
elements:
\begin{enumerate}
\item $\pm I$;
\item $\theta=\begin{pmatrix} -1& -1\cr 1& 0\cr\end{pmatrix}$, which has order $3$;
\item $-\theta$, which has order $6$;
\item $\epsilon=\begin{pmatrix} -1& 0\cr 0&  1\cr\end{pmatrix}$ or $\tau=\begin{pmatrix} 0& 1\cr 1& 0\cr\end{pmatrix}$, which have order $2$ and are not conjugate to each other,
\item $\epsilon\tau$, which has order $4$;
\item $U(n)=\begin{pmatrix} 1& n\cr 0&  1\cr\end{pmatrix}$ with $n>0$;
\item $-U(n)$ with $n>0$.
\end{enumerate}
\end{proposition}

From the obvious presentation $\G_\phi = \<a,b,t \mid tat^{-1}=\phi(a),\ tat^{-1}=\phi(b)\>$
we get the presentation 
$$
H_1(\G_\phi,\Z) = \<a,b,t \mid [a,b]=[a,t]=[b,t]=1 = a^{-1}\phi(a)=b^{-1}\phi(b)\>$$
from which it is easy to calculate the following.

\begin{lemma}\label{l:H1}
With the notation of Proposition \ref{p:types}:
\begin{enumerate}
\item $H_1(\G_I,\Z) \cong\Z^3$ and $H_1(\G_{-I},\Z) \cong \Z\oplus \Z/2\oplus \Z/2$;
\item $H_1(\G_{\theta},\Z) \cong \Z\oplus \Z/3$;
\item $H_1(\G_{-\theta},\Z) \cong \Z$;
\item $H_1(\G_{\epsilon},\Z) \cong \Z^2\oplus \Z/2$  and $H_1(\G_{\tau},\Z) \cong \Z^2$;
\item $H_1(\G_{\epsilon\tau},\Z) \cong \Z\oplus \Z/2$;
\item $H_1(\G_{U(n)},\Z) = \Z^2\oplus \Z/n$ if $n>0$;
\item $H_1(\G_{-U(n)},\Z) = \Z\oplus \Z/4 $ if $n$ odd, and $H_1(\G_{-U(n)}) = \Z\oplus \Z/2\oplus\Z/2$ if $n$ even.
\end{enumerate}
\end{lemma}

\end{appendices}

 \bibliographystyle{plain}

 \Addresses

\end{document}